\documentclass[10pt,a4paper]{article}

\usepackage[pdftex]{graphicx}
\usepackage{amsmath}
\usepackage{amssymb}
\usepackage{amsthm}
\usepackage{amsfonts}
\usepackage{wrapfig}
\usepackage{calc}
\usepackage{tabulary}
\usepackage[plain]{fullpage}
\usepackage{xcolor}
\usepackage{tikz}
\usetikzlibrary{snakes}

\newcommand{\aut}{\textnormal{Aut}}

\newcommand{\disj}{\textnormal{disj}}
\newcommand{\supp}{\textnormal{supp}}

\newcommand{\path}[1]{\textnormal{Path}\langle {#1} \rangle}
\newcommand{\dinfinity}{D_\infty}

\newcommand{\indec}{\mathrm{Indec}}

\makeatletter
\providecommand*{\cupdot}{%
  \mathbin{%
    \mathpalette\@cupdot{}%
  }%
}
\newcommand*{\@cupdot}[2]{%
  \ooalign{%
    $\m@th#1\cup$\cr
    \sbox0{$#1\cup$}%
    \dimen@=\ht0 %
    \sbox0{$\m@th#1\cdot$}%
    \advance\dimen@ by -\ht0 %
    \dimen@=.5\dimen@
    \hidewidth\raise\dimen@\box0\hidewidth
  }%
}

\providecommand*{\bigcupdot}{%
  \mathop{%
    \vphantom{\bigcup}%
    \mathpalette\@bigcupdot{}%
  }%
}
\newcommand*{\@bigcupdot}[2]{%
  \ooalign{%
    $\m@th#1\bigcup$\cr
    \sbox0{$#1\bigcup$}%
    \dimen@=\ht0 %
    \advance\dimen@ by -\dp0 %
    \sbox0{\scalebox{2}{$\m@th#1\cdot$}}%
    \advance\dimen@ by -\ht0 %
    \dimen@=.5\dimen@
    \hidewidth\raise\dimen@\box0\hidewidth
  }%
}
\makeatother

\newtheorem{lemma}{Lemma}[section]
\newtheorem{theorem}[lemma]{Theorem}

\newtheorem{example}[lemma]{Example}
\newtheorem{dfn}[lemma]{Definition}
\newtheorem{prop}[lemma]{Proposition}

\newtheorem{question}[lemma]{Question}

\title{The Reconstruction of Cycle-free Partial Orders from their Abstract Automorphism Groups III : Decorated CFPOs}

\author{Robert Barham \\ Institut f\"ur Algebra, TU Dresden \\ robert.barham@yahoo.co.uk}

\begin{document}

\maketitle

\thanks{The author has received funding from the European Research Council under the European Community's Seventh Framework Programme (FP7/2007-2013 Grant Agreement no. 257039).}

\abstract{In this triple of papers, we examine when two cycle-free partial orders can share an abstract automorphism group.  This question was posed by M. Rubin in his memoir concerning the reconstruction of trees.

In this final paper, we give describe a way of constructing `decorated' CFPOs by attaching treelike CFPOs to and between the elements of a cone transitive CFPO.  We then show that the automorphism groups of the components of of a decorated CFPO are second order definable in the abstract automorphism group of the decorated CFPO.}

\section{Introduction}

The two types of CFPOs we considered in the previous two parts are very different in character, so it seems reasonable that perhaps out two methods can be combined in some way, without too much interference.  This is indeed possible, and in this paper we will combine treelike and members of $K_{Cone}$ in such a way that the automorphism groups of the components are definable in the whole automorphism group, and so our previous reconstruction results will be applicable.

Section \ref{section:dec} will give the method of decoration and describe the resulting automorphism groups as wreath products of the automorphism groups of the components, while Section \ref{sec:interpretation} will define these components using second order logic.  This is a desirable outcome, because if the components are definable, then we can perform our interpretations inside the definable sets rather than the whole group, reconstructing the component structures.

In Section \ref{sec:final} there are some concluding remarks, and some open problems raised by these papers.

\section{Decoration}\label{section:dec}

We will first look at attaching trees above and between points of a member of $K_{Cone}$, and give conditions for when a general CFPO shares an automorphism group with such a CFPO.

\begin{dfn}
If $M$ is a CFPO then we define $M_{ap}$ to be the set
$$\lbrace (i,j) \in M^2 \: : \: i <_M j \wedge \forall k \in M \neg (i <_M k <_M j)  \rbrace$$
ap stands for `adjacent pairs'.
\end{dfn}

\begin{dfn}
Let$\langle M, \leq_M \rangle$ be a CFPO and let $\langle S, \leq_S \rangle$ and  $\langle T, \leq_T, L \rangle$ be trees, where $L$ is a unary predicate that picks out a maximal chain of $T$.  The structure $\mathrm{Dec}(M,S,(T,L))$ is the partial order with universe
$$|M| \cupdot \bigcupdot\limits_{i \in M} |S_i| \cupdot \bigcupdot\limits_{(i,j) \in M_{ap}} |T_{(i,j)}|$$
where:
\begin{itemize}
\item $S_i \cong S$ for every $i \in M$
\item $T_{(i,j)} \cong T$ for every $(i,j) \in \mathcal{M}$.  We use $L_{(i,j)}$ to denote the maximal chain of $T_{(i,j)}$ picked out by $L$.
\end{itemize}
$\mathrm{Dec}(M,S,(T,L))$ is ordered by $\leq_D$, which is the transitive closure of the following:\vspace*{-10pt}

$$\begin{array}{ c  r}
x \leq_M y & \quad \mathrm{or} \\
x \leq_{S_i} y  & \mathrm{or} \\
y \in S_x & \mathrm{or} \\
x \leq_{T_{(i,j)}} y  & \mathrm{or} \\
\exists z \in M \: L_{(x,z)}(y) & \mathrm{or} \\
\exists z \in M \: L_{(z,y)}(x)
\end{array}
$$

Informally, we attach a copy of $S$ above every point of $M$, and glue a copy of $T$ between every adjacent pair of $M$ along $L$.
\end{dfn}

Note that if $M_{ap}$ is empty, in other words if $M$ is dense, then $$\mathrm{Dec}(M,S,(T_0,L_0)) = \mathrm{Dec}(M,S,(T_1,L_1))$$ for all $(T_0,L_0)$ and $(T_1,L_1)$.

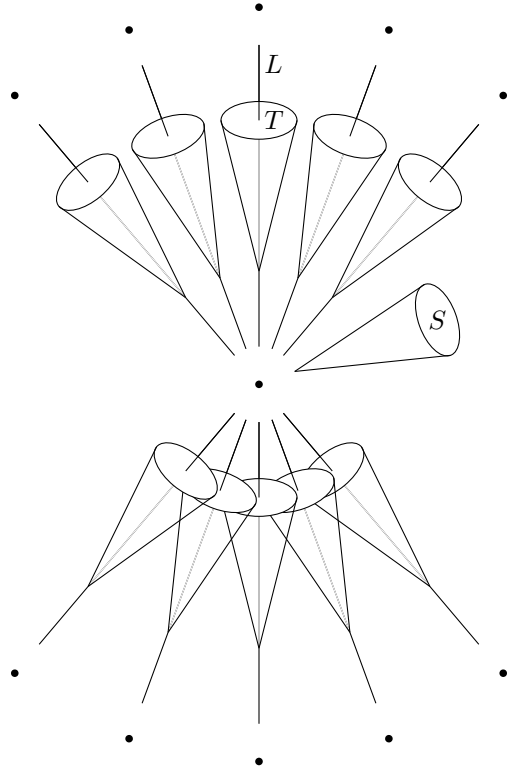
\begin{figure}[!b]\label{Fig:VagueDecoration}
\begin{center}
\begin{tikzpicture}[scale=0.05]
\fill (0,0) circle (1);
\draw (0,10) -- (0,90);
\fill[white] (0,70) ellipse [x radius=10,y radius=5];
\draw[white] (0,30) -- (0,70);
\draw (0,30) -- (-10,70);
\draw (0,30) -- (10,70);
\draw (0,70) ellipse [x radius=10,y radius=5];
\draw (0,70) -- (0,90);
\fill (0,100) circle (1);
\draw[anchor=west] (-1,85) node {$L$};
\draw[anchor=west] (-1,70) node {$T$};

\draw[rotate=20] (0,10) -- (0,90);
\fill[white,rotate=20](0,70) ellipse [x radius=10,y radius=5];
\draw[white,rotate=20] (0,30) -- (0,70);
\draw[rotate=20] (0,30) -- (-10,70);
\draw[rotate=20] (0,30) -- (10,70);
\draw[rotate=20] (0,70) ellipse [x radius=10,y radius=5];
\draw[rotate=20] (0,70) -- (0,90);
\fill[rotate=20] (0,100) circle (1);

\draw[rotate=40] (0,10) -- (0,90);
\fill[white,rotate=40](0,70) ellipse [x radius=10,y radius=5];
\draw[white,rotate=40] (0,30) -- (0,70);
\draw[rotate=40] (0,30) -- (-10,70);
\draw[rotate=40] (0,30) -- (10,70);
\draw[rotate=40] (0,70) ellipse [x radius=10,y radius=5];
\draw[rotate=40] (0,70) -- (0,90);
\fill[rotate=40](0,100) circle (1);

\draw[rotate=-20] (0,10) -- (0,90);
\fill[white,rotate=-20](0,70) ellipse [x radius=10,y radius=5];
\draw[white,rotate=-20] (0,30) -- (0,70);
\draw[rotate=-20] (0,30) -- (-10,70);
\draw[rotate=-20] (0,30) -- (10,70);
\draw[rotate=-20] (0,70) ellipse [x radius=10,y radius=5];
\draw[rotate=-20] (0,70) -- (0,90);
\fill[rotate=-20] (0,100) circle (1);

\draw[rotate=-40] (0,10) -- (0,90);
\fill[white,rotate=-40](0,70) ellipse [x radius=10,y radius=5];
\draw[white,rotate=-40] (0,30) -- (0,70);
\draw[rotate=-40] (0,30) -- (-10,70);
\draw[rotate=-40] (0,30) -- (10,70);
\draw[rotate=-40] (0,70) ellipse [x radius=10,y radius=5];
\draw[rotate=-40] (0,70) -- (0,90);
\fill[rotate=-40] (0,100) circle (1);

\draw[rotate=40] (0,-90) -- (0,-10);
\fill[white,rotate=40](0,-30) ellipse [x radius=10,y radius=5];
\draw[white,rotate=40] (0,-70) -- (0,-30);
\draw[rotate=40] (0,-70) -- (-10,-30);
\draw[rotate=40] (0,-70) -- (10,-30);
\draw[rotate=40] (0,-30) ellipse [x radius=10,y radius=5];
\draw[rotate=40] (0,-30) -- (0,-10);
\fill[rotate=40] (0,-100) circle (1);

\draw[rotate=20] (0,-90) -- (0,-10);
\fill[white,rotate=20](0,-30) ellipse [x radius=10,y radius=5];
\draw[white,rotate=20] (0,-70) -- (0,-30);
\draw[rotate=20] (0,-70) -- (-10,-30);
\draw[rotate=20] (0,-70) -- (10,-30);
\draw[rotate=20] (0,-30) ellipse [x radius=10,y radius=5];
\draw[rotate=20] (0,-30) -- (0,-10);
\fill[rotate=20] (0,-100) circle (1);

\draw[rotate=0] (0,-90) -- (0,-10);
\fill[white,rotate=0](0,-30) ellipse [x radius=10,y radius=5];
\draw[white,rotate=0] (0,-70) -- (0,-30);
\draw[rotate=0] (0,-70) -- (-10,-30);
\draw[rotate=0] (0,-70) -- (10,-30);
\draw[rotate=0] (0,-30) ellipse [x radius=10,y radius=5];
\draw[rotate=0] (0,-30) -- (0,-10);
\fill[rotate=0] (0,-100) circle (1);

\draw[rotate=-20] (0,-90) -- (0,-10);
\fill[white,rotate=-20](0,-30) ellipse [x radius=10,y radius=5];
\draw[white,rotate=-20] (0,-70) -- (0,-30);
\draw[rotate=-20] (0,-70) -- (-10,-30);
\draw[rotate=-20] (0,-70) -- (10,-30);
\draw[rotate=-20] (0,-30) ellipse [x radius=10,y radius=5];
\draw[rotate=-20] (0,-30) -- (0,-10);
\fill[rotate=-20] (0,-100) circle (1);

\draw[rotate=-40] (0,-90) -- (0,-10);
\fill[white,rotate=-40](0,-30) ellipse [x radius=10,y radius=5];
\draw[white,rotate=-40] (0,-70) -- (0,-30);
\draw[rotate=-40] (0,-70) -- (-10,-30);
\draw[rotate=-40] (0,-70) -- (10,-30);
\draw[rotate=-40] (0,-30) ellipse [x radius=10,y radius=5];
\draw[rotate=-40] (0,-30) -- (0,-10);
\fill[rotate=-40] (0,-100) circle (1);

\fill[white,rotate=-70] (0,50) ellipse [x radius=10,y radius=5];
\draw[rotate=-70] (0,10) -- (-10,50);
\draw[rotate=-70] (0,10) -- (10,50);
\draw[rotate=-70] (0,50) ellipse [x radius=10,y radius=5];
\draw[rotate=-70] (0,50) node {$S$};
\end{tikzpicture}
\end{center}
\caption{A vague illustration of Decoration}
\end{figure}

\begin{example}\label{ExampleDec}
An illustration of the neighbourhood of an element of $M$ in $\mathrm{Dec}(M,S,(T,L))$ is given in Figure 32.  A more specific example of decorating is pictured in Figure 33.  In this example, we do not need to specify an $L$, as $B$ has exactly one maximal chain.
\end{example}

\begin{figure}\label{Fig:SpecificDecoration}
\begin{center}
\begin{tikzpicture}[scale=0.09]
\draw (-80,10) -- (-70,20) -- (-60,10);
\draw (-80,30) -- (-70,20) -- (-60,30);

\fill (-80,10) circle (0.5);
\fill (-70,20) circle (0.5);
\fill (-60,10) circle (0.5);
\fill (-80,30) circle (0.5);
\fill (-60,30) circle (0.5);

\draw (-78,20) node[anchor=east] {$A=$};

\fill (-30,20) circle (0.5);
\draw (-30,20) node[anchor=east] {$B=$};

\draw (12,20) node[anchor=east] {$\mathrm{Dec}(A,B,B)=$};

\draw (10,10) -- (10,0) -- (20,15) -- (30,0) -- (30,10);
\draw (10,40) -- (10,30) -- (20,15) -- (30,30) --(30,40);

\fill (10,10) circle (0.5);
\fill (20,15) circle (0.5);
\draw (20,15) -- (20,25);
\fill (20,25) circle (0.5);

\fill (30,10) circle (0.5);
\fill (10,30) circle (0.5);
\fill (30,30) circle (0.5);
\fill (10,40) circle (0.5);
\fill (30,40) circle (0.5);
\fill (10,0) circle (0.5);
\fill (30,0) circle (0.5);

\fill (15,7.5) circle (0.5);
\fill (25,7.5) circle (0.5);
\fill (15,22.5) circle (0.5);
\fill (25,22.5) circle (0.5);

\end{tikzpicture}
\end{center}
\caption{Example \ref{ExampleDec}}
\end{figure}

\begin{prop}
$\mathrm{Dec}(M,S,(T,L))$ is a CFPO for any $M$, $S$ and $(T,L)$.
\end{prop}
\begin{proof}
Let $a$ and $b$ be such that there are two different paths between them, which we will call $P_0$ and $P_1$.  If $a$ and $b$ are contained in the same copy of $S$ or $(T,L)$ then this contradicts our assumption that $S$ and $(T,L)$ are trees.  If $a \in S_{m_a}$ then $\lbrace m_a \rbrace \subseteq P_0 \cap M \cap P_1$.  If $a \in T_{(m_a,m'_a)}$ then either $m_a$ or $m'_a$ is in $P_0 \cap M$.  Similarly either $m_a$ or $m'_a$ is in $P_1 \cap M$.

Thus the starting point of $P_0 \cap M$ is one of $a$, $m_a$ or $m'_a$, while the ending point is one of $b$, $m_b$ and $m'_b$.  The same conclusion can be reached for $P_1 \cap M$.  If $P_0 \cap M$ starts with $m_a$ while $P_1 \cap M$ starts with $m'_a$ then either $P_0$ or $P_1$ has to pass through the starting point of the other, which implies that one of the paths doubles back on itself, giving a contradiction.  Since $P_0 \cap M$ and $P_1 \cap M$ have the same start and end points, the fact that $M$ is a CFPO implies that they must be equal.

Then $P_0$ and $P_1$ `move through'  $M$ in the same way, and so must differ by their behaviour within the copies of $S$ and $(T,L)$.  But both $S$ and $(T,L)$ are trees, so have unique paths and therefore $P_0=P_1$.
\end{proof}

\begin{lemma}\label{preservesX}
$\aut(\mathrm{Dec}(M,S,(T,L)))$ preserves $M$ setwise.
\end{lemma}
\begin{proof}
Since $M \in K_{Cone}$, given any $a \in M$ there are $b_0,b_1 \in M$ such that $b_0 || b_1$ and $a = b_0 \vee b_1$.  Let $\phi \in \aut(\mathrm{Dec}(M,S,(T,L))$.  Since $\phi$ is an automorphism $\phi(a)= \phi(b_0) \vee \phi(b_1)$.

$S$ and $(T,L)$ are trees, so $ \phi(b_0) \vee \phi(b_1)$ cannot be contained in a copy of $S$ or $(T,L)$, and so all automorphisms of $\mathrm{Dec}(M,S,(T,L))$ preserve $M$.
\end{proof}

\begin{theorem}\label{thm:KDecwidth}
Let $M$ be a CFPO, let $A$ be a 1-orbit such that $\aut(M)$ acts cone transitively on $A$, and for any $B \subset M$ let $ \sim_B$ be the equivalence relation $x \sim y \Leftrightarrow \path{x,y} \cap B = \emptyset$.  We let $C \in (M \setminus A) / \sim_A$, and describe two conditions.
\begin{enumerate}
\item If $\path{C,M \setminus C} \not= \emptyset$ then there is an $a_c \in A$ such that $$\path{C,M \setminus C} = \lbrace a_C \rbrace$$

This says that if there is only one way to go from $C$ to $M \setminus C$ then $C$ is attached to $a_c$.
\item If $\path{C,M \setminus C} = \emptyset$ then:
\begin{enumerate}
\item $(M \setminus C) / \sim_C$ has exactly two elements which we call $B_C$ and $B'_C$; and
\item there is $(a_C, a'_C) \in A_{ap}$ such that
$$
\begin{array}{c c c}
\path{C,B_C} = \lbrace a_C \rbrace & \textnormal{ and } & \path{C,B'_C} = \lbrace a'_C \rbrace
\end{array}
$$
\end{enumerate}
This says that if there is more than one way to go from $C$ to $M \setminus C$ then $C$ lies between an adjacent pair of $A$.
\end{enumerate}
If every $C \in (M \setminus A) / \sim_A$ satisfy both 1. and 2. then there are trees $S$ and $(T,L)$ and a cone transitive CFPO $X$ such that
$$\aut(M) \cong \aut(\mathrm{Dec}(X,S,(T,L))$$
\end{theorem}
\begin{proof}
Suppose $M$ has a 1-orbit $A$ that satisfies the conditions of the theorem.  We define $X$ to be the substructure of $M$ with domain $A$.

We define the following set:
$$\mathcal{C}_S := \lbrace C \in (M \setminus A) / \sim_A \: : \: \path{C,M\setminus C} \not= \emptyset \rbrace$$
and let $C \in \mathcal{C}_S$.  We wish to show that $C$, when acted on by $\aut_{\lbrace C \rbrace}(M)$, is treelike.  If $C$ does not embed Alt then $C$, even with its full automorphism group, is treelike, so we suppose that $C$ does embed Alt, which we enumerate as $( \ldots c_{-1}, c_0, c_1, \ldots )$.  There must be some $i$ such that for all $j$
$$
\begin{array}{ c c c}
\path{a_C,c_i} \subseteq \path{a_C,c_j} & \textnormal{ or } & \path{a_C,c_{i+1}} \subseteq \path{a_C,c_j}
\end{array}
$$
If $\phi \in \aut(M)$ and $i \not= j$ then $\phi$ cannot map $c_i$ to $c_j$, otherwise
$$\path{a_C,a_i} \cap \path{\phi(a_C),a_j} = \emptyset$$
which contradicts our assumption that $\path{C,M \setminus C} \not= \emptyset$.  Thus $\aut_{\lbrace C \rbrace }(M)$ cannot act as $\dinfinity$ on any copy of Alt that is contained in $C$, so $C$ with the action of $\aut_{\lbrace C \rbrace}(M)$ is treelike (Theorem 5.13 of Part 1), and we let $\langle S_C , \leq_{C} \rangle$ be a tree with the action of $\aut(M)$.

Pick any $a \in A$ and let $\lbrace C \in \mathcal{C}_S \: : \: \path{C,M \setminus C} = \lbrace a \rbrace \rbrace$ be enumerated by $(C_i \: : \: i \in I)$.  We define $S$ to be the tree with domain
$$\lbrace r \rbrace \cup \bigcup_{i \in I} S_{C_i}$$
and order
$$x \leq_S y \textnormal{ iff } \left\lbrace
\begin{array}{ l  r}
x = r & \quad \mathrm{or} \\
x \leq_{C_i} y 
\end{array}
\right.$$
$S$ is independent of our choice of $a$ because $A$ is an orbit.

To find $T$, we define
$$\mathcal{C}_T :=  \lbrace C \in (M \setminus A) / \sim_A \: : \: \path{C,M\setminus C} = \emptyset \rbrace$$
Note that if $C,D \in \mathcal{C}_T$ are such that $a_C = a_D$ and $a'_C=a'_D$ then $C=D$, as there is a path from $a_C$ to $a'_C$ contained in both $C$ and $D$.

Let $C \in \mathcal{C}_T$.  Any automorphism of $M$ that fixes $C$ must also fix $a_C$ and $a'_C$, and hence fixes $\path{a_C,a'_C}$ set-wise, so we may introduce a unary predicate $L$ which is realised exactly on $\path{a_C,a'_C}$.  We also use the symbol $L$ to denote $\path{a_C,a'_C}$.  Since a path cannot embed Alt, the set of points realising $L$ is treelike, and indeed the resulting tree is a linear order, which we call $L_T$ with ordering $\leq_L$.

Note that each of the elements of $(C \setminus L) / \sim_L$ is also treelike, for the same reasons that the members of $\mathcal{C}_S$ are treelike.  We enumerate the equivalence classes of $C \setminus L$ as $(D_j \: : \: j \in J)$, denote the tree which correspond to $D_j$ by $T_j$, and for each $j$ we partition $L$ into
$$
\begin{array}{l c}
L'_j := \lbrace l \in L \: : \: l \in \path{D_j,a_C} \rbrace & \textnormal{and} \\
 L_j := \lbrace l \in L \: : \: l \in \path{D_j,a'_C} \rbrace \setminus L'_j
\end{array}
$$
Finally we are in a position to define our candidate for $(T,L)$.  The domain is
$$ L_T \cup \bigcup_{j \in J} T_j $$
while the ordering is:
$$x \leq_T y \Leftrightarrow \left\lbrace
\begin{array}{ c  r}
x \leq_{L} y & \quad \mathrm{or} \\
x \leq_{T_j} y & \mathrm{or} \\
y \in T_i \textnormal{ and } x \in L_j 
\end{array}
\right.$$
and the predicate $L$ is carried across from $C$.  The $(T,L)$ are independent of our choice of element from $A_{ap}$ as $\aut(M)$ acts cone transitively on $A$.

We now have candidates for $X$, $S$ and $(T,L)$.

Given $\phi \in \aut(\mathrm{Dec}(X,S,(T,L))$ we seek to show how that $\phi$ can be viewed as an automorphism of $M$.  Since $\phi$ preserves $X$ setwise (Lemma \ref{preservesX}), it preserves $A$.

$\aut(M)$ acts cone transitively on $A$, so given any two $x, y \in A$ there is an automorphism of $M$ that maps $x$ to $y$, hence mapping $\lbrace C \in \mathcal{C}_S \: : \:\path{C,M\setminus C}= \lbrace x \rbrace \rbrace $ to $\lbrace C \in \mathcal{C}_S \: : \:\path{C,M\setminus C}= \lbrace y \rbrace \rbrace $.  Therefore
$$\bigcup \lbrace C \in \mathcal{C}_S \: : \:\path{C,M\setminus C}= \lbrace x \rbrace \rbrace \cong \bigcup \lbrace C \in \mathcal{C}_S \: : \:\path{C,M\setminus C}= \lbrace y \rbrace \rbrace $$
By construction
$$\aut(S) \cong_A \aut(\bigcup \lbrace C \in \mathcal{C}_S \: : \:\path{C,M\setminus C}= \lbrace x \rbrace \rbrace) $$
So $\phi$ acts as an automorphism of $\bigcup \mathcal{C}_S $.

$\aut(M)$ acts cone transitively on $A$, so given any two $(x_0, y_0), (x_1,y_1) \in A_{ap}$ there is an automorphism of $M$ that maps $(x_0, y_0)$ to $(x_1,y_1)$.  Each $C \in \mathcal{C}_T$ is uniquely determined by if $a_C$ and $a_C'$ therefore if $C_0, C_1 \in \mathcal{C}_T$ then $C_0 \cong C_1$.
By construction, for all $C \in \mathcal{C}_T$
$$\aut(T) \cong_A \aut(C)$$
So $\phi$ acts as an automorphism of $\bigcup \mathcal{C}_T$.

Therefore every automorphism of $\mathrm{Dec}(X,S,(T,L))$ is also an automorphism of $M$.

If $\phi$ is an automorphism of $M$ then it preserves $A$, and thus $X$, and since every element of $\mathcal{C}_S$ and $\mathcal{C}_T$ is isomorphic to $S$ or $T$ respectively, it is also an automorphism of $\mathrm{Dec}(X,S,(T,L))$.
\end{proof}

\begin{dfn}\label{Dfn:WreathProduct}
Given an abstract group $G$ and a permutation group $(H,S,\mu(h,s))$ their wreath product, written as $G \wr_{S} H$, is the abstract group on domain
$$H \times \lbrace \eta : S \rightarrow G \rbrace$$
We use $\eta(s)$ to denote the function $s \mapsto \eta(s)$, and so $\eta(s_0 s)$ is the function $s \mapsto \eta(s_0 s)$.  The group operation of $G \wr_{S} H$ is given by
$$(h_0, \eta_0(x))(h_1, \eta_1(x))=(h_0 h_1, \eta_0(\mu (h_1^{-1},x)) \eta_1(x))$$
\end{dfn}

\begin{dfn}
Let $X \in K_{Cone}$ and let $S,(T,L) \in K_{Rub}$.  We introduce the notation
$$W(X,S,(T,L)) : = \aut((T,L)) \wr_{X_{ap}} (\aut(S) \wr_{X} \aut(X))$$
where the action of $\aut(S) \wr_X \aut(X)$ on $X_{ap}$ is given by
$$(\phi, \eta) (x,y) = (\phi(x),\phi(y))$$
\end{dfn}

When only one $W(X,S,(T,L))$ is being discussed, we may denote it as $W$ for brevity.

\begin{prop}
Let $X$ be a cone transitive CFPO and let $S$ and $(T,L)$ be trees.
$$\aut(\mathrm{Dec}(X,S,(T,L)) \cong W(X,S,(T,L))$$
\end{prop}
\begin{proof}
Even through we regard $W(X,S,(T,L))$ as an abstract group, it has a natural action on $\mathrm{Dec}(X,S,(T,L))$, which we will call $\mu$.  We introduce the notation $I^x_y$ for the identity map from $S_x$ to $S_y$, and $I^{(x,y)}_{(w,z)}$ for the identity map from $T_{(x,y)}$ to $T_{(w,z)}$, and define $\mu$ as follows:
$$
\mu ((\phi, \eta, \zeta),x) = \left\lbrace
\begin{array}{r c l}
(\phi(x)) & \textnormal{ if } & x \in X \\
I^\alpha_{\phi(\alpha)}(\eta(\alpha)(x)) & \textnormal{ if } & x \in S_\alpha \\
 I^{(\alpha,\beta)}_{\phi(\alpha,\beta)}(\zeta((\alpha,\beta))(x)) & \textnormal{ if } & x \in T_{(\alpha,\beta)}
\end{array}\right.
$$
For any $\phi$, $\eta$ and $\zeta$ the function $x \mapsto \mu ((\phi, \eta, \zeta),x)$ is an automorphism, as $\phi$ is an automorphism and for every $\alpha$ and $\beta$ both
$$
\begin{array}{c c c}
I^\alpha_{\phi(\alpha)}(\eta(\alpha)(x)): S_\alpha \rightarrow S_{\phi(\alpha)} & \mathrm{and} & I^{(\alpha,\beta)}_{\phi(\alpha,\beta)}(\zeta((\alpha,\beta))(x)) : T_{(\alpha,\beta)} \rightarrow T_{\phi(\alpha,\beta)}
\end{array}
$$
are isomorphisms.  Additionally, each $(\phi, \eta, \zeta)$ results in a unique automorphism.  To see this suppose for a contradiction that
$$\forall x  \: \mu ((\phi_0, \eta_0, \zeta_0),x) = \mu ((\phi_1, \eta_1, \zeta_1),x)$$
Since this is for all $x$, it is true for all $x \in X$ in particular, and thus $\phi_0 = \phi_1$.  We also have
$$
\begin{array}{c c c}
\forall x \in S_\alpha \: \eta_0(\alpha)(x)=\eta_1(\alpha)(x) & \mathrm{and} & \forall x \in T_{(\alpha,\beta)} \: \zeta_0((\alpha,\beta))(x)=\zeta_1((\alpha,\beta))(x)
\end{array}
$$
and thus $\eta_0 = \eta_1$ and $\zeta_0 = \zeta_1$.  Finally, if we are able to show that every automorphism of $\mathrm{Dec}(X,S,(T,L))$ can be represented in this way, we will have proved this proposition.

Let $\psi$ be an automorphism of $\mathrm{Dec}(X,S,(T,L))$.  We set $\phi := \psi |_{X}$ and we set the function components as follows:
$$
\begin{array}{c c c}
\eta(\alpha) = \psi |_{S_{\alpha}} & \mathrm{and} & \zeta((\alpha,\beta)) = \psi |_{T_{(\alpha,\beta)}}
\end{array}
$$
Which gives an element of $W(X,S,(T,L))$ whose action on $\mathrm{Dec}(X,S,(T,L))$ via $\mu$ is the same as $\psi$.  Thus the map
$$
\begin{array}{r c l}
W(X,S,(T,L)) & \rightarrow & \aut(\mathrm{Dec}(X,S,(T,L))) \\
(\phi, \eta, \zeta) & \mapsto & \mu((\phi,\eta,\zeta),x)
\end{array}
$$
is bijective and, since $\mu$ is a group action, an isomorphism.
\end{proof}

\section{Interpreting Inside a Wreath Product}\label{sec:interpretation}

When we interpreted $M \in K_{Cone}$ inside its automorphism group, we made use of the subgroups isomorphic to $A_5$.  These subgroups still exist in the automorphism groups of the CFPOs we obtained through decoration, as $\aut(X) \leq W(X,S,(T,L))$.

If we can adapt the interpretation so that it ignores the decoration then we will be able to recover $X$.  Subsection \ref{subsection:reconX} works towards this by adding in a few clauses to the formulas of Part II.  Subsection \ref{subsection:reconSTL} gives second-order formulas that define subgroups of $W(X,S,(T,L))$ isomorphic to $\aut(S)$ and $\aut(T,L)$.

\subsection{Reconstructing $X$}\label{subsection:reconX}

\begin{lemma}
Recall that $A_5(\bar{f})$ is the formula that states that $\bar{f}$ satisfies the elementary diagram of $A_5$.  If $W \models A_5(\bar{f})$ then $\bar{f}$ fixes an element of $X \subset \mathrm{Dec}(X,S,(T,L))$.
\end{lemma}
\begin{proof}
The automorphisms of $\mathrm{Dec}(X,S,(T,L))$ preserve $X$ (Lemma 2.5), so if $\bar{f}|_X \not= id$ then $\bar{f}$ has a fixed point in $X$ by Lemma 3.2 of Part II.  If $\bar{f}|_X = id$ then $\bar{f}$ fixes $X$.
\end{proof}

\begin{lemma} Many of the formulas in Chapter 4 retain either their exact meaning, or something very similar, in  $W(X,S,(T,L))$, which we call $W$.
\begin{enumerate}
\item If $W \models \mathrm{Indec}(\bar{f})$ then $ \bigcup\limits_{x,y \in \supp(\bar{f})} \path{x,y} \setminus \supp(\bar{f}) $ is a singleton, which we call $f$.
\item If $W \models \mathrm{Disj}(\bar{f}, \bar{g})$ then the support of $\bar{f}$ and $\bar{g}$ are disjoint.
\item If $W \models [ \supp(\bar{f}) \sqsubset \supp(\bar{g}) ]$ then $\supp(\bar{f}) \subset \supp(\bar{g})$.
\item If $W \models \mathrm{SamePD}(\bar{f}, \bar{g})$ then either:
\begin{itemize}
\item $f = g$,
\item $f \in X$ and $g \in S_f$ or $g \in T_{(f,h)}$ for some $h$, or
\item $g \in X$ and $f \in S_f$ or $f \in T_{(g,h)}$ for some $h$.
\end{itemize} 
\end{enumerate}
\end{lemma}
\begin{proof}
Note that for all $\phi \in W$ if $x \in \supp(\phi)$ then $S_x, T_{(x,y)} \subset \supp(\phi)$ for all $y$.
\begin{enumerate}
\item Suppose $W \models \mathrm{Indec}(\bar{f})$.  If $\bar{f}|_X \not= id$ then the singleton we found in Lemma 3.2 of Part II (unique by Proposition 3.12 of Part II) works in this context.

Suppose that $\bar{f}|_X = id$.  Since $\bar{f}$ is indecomposable, then either $\supp(\bar{f}) \subseteq S_x$ or $\supp(\bar{f}) \subseteq T_{(x,y)}$ for some $x \in X$ or $(x,y) \in X_{ap}$.  We define $$x_f := \path{\supp(\bar{f}, \mathrm{Dec}(X,S,(T,L)) \setminus \supp(\bar{f})}$$
\item The proof of Lemma 3.13 of Part II does not require serious adaptation for this context.
\item If $W \models [ \supp(\bar{f}) \sqsubset \supp(\bar{g}) ]$ and at least one of $x_f$ and $x_g$ is in $W \setminus X$ then either $\supp(\bar{f}) \subset \supp(\bar{g})$ or $\supp(\bar{f}) \subset \supp(\bar{g})$, and the argument in the appropriate case of the proof of Lemma 3.14 of Part II suffices.
\item If both $x_f$ and $x_g$ are contained in $X$ then Lemma 3.17 of Part II shows that $x_f = x_g$.  If both $x_f$ and $x_g$ are in $W \setminus X$ then the proof of Lemma 3.17 of Part II shows that $x_f = x_g$.

Suppose that $x_f \in X$ and $x_g \in W \setminus X$.  If $x_g \in S_y$  or $x_g \in T_{(y,y')}$ for $y \not= x_f$ then the same witness that observes that $W \models \neg\mathrm{SamePD}(\bar{f}, \bar{h})$ shows that $W \models \neg\mathrm{SamePD}(\bar{f}, \bar{g})$, so $x_g \in S_f$ or $x_g \in T_{(x_f,y)}$ for some $y$.

Similarly, if $x_g \in X$ and $x_f \in W \setminus X$ then $x_f \in S_f$ or $x_f \in T_{(x_g,y)}$ for some $y$.
\end{enumerate}
\end{proof}

\begin{dfn}
Let $\mathrm{MeetsX}(\bar{f})$ be the following formula

$$\indec(\bar{f}) \wedge \exists \bar{g} \left( \begin{array}{c}
\neg\disj(\bar{f},\bar{g}) \wedge  \neg \mathrm{SamePD}(\bar{f},\bar{g})) \wedge \\
 \neg [\supp(\bar{f}) \sqsubset \supp(\bar{g})] \wedge \neg [\supp(\bar{g}) \sqsubset \supp(\bar{f})]
\end{array} \right)$$
\end{dfn}

\begin{lemma}
$ W \models \mathrm{MeetsX}(\bar{f})$ if and only if $\supp(\bar{f}) \cap X \not= \emptyset$.
\end{lemma}
\begin{proof}
First we suppose for a contradiction that both $ W \models \mathrm{MeetsX}(\bar{f})$ and $\supp(\bar{f}) \cap X = \emptyset$.  Since $W \models \neg \mathrm{SamePD}(\bar{f},\bar{g}))$, we know that $x_f \not= x_g$.

Since $\supp(\bar{f}) \cap X = \emptyset$, the point $x_f$ must be contained in one of the trees that we decorated $X$ with, so $W \models \neg \disj(\bar{f},\bar{g}) $ implies that $\supp(\bar{f}) \subseteq \supp(\bar{g})$ or $\supp(\bar{g}) \subseteq \supp(\bar{f})$, giving a contradiction.

Suppose $\supp(\bar{f}) \cap X \not= \emptyset$.  We can find a $\bar{g}$ to witness $ W \models \mathrm{MeetsX}(\bar{f})$ by taking any tuple that fixes a point inside $\supp(\bar{f})$ which moves $x_f$.
\end{proof}

\begin{dfn}
Let $\mathrm{RepPointDec}(\bar{f}_0, \bar{f}_1)$ be the formula
$$
\begin{array}{c}
\disj(\bar{f}_0,\bar{f}_1) \wedge \mathrm{MeetsX}(\bar{f}_0) \wedge \mathrm{MeetsX}(\bar{f}_1) \wedge\\
\forall \bar{g} \exists \bar{h} \left( \left(
\begin{array}{r}
\mathrm{MeetsX}(\bar{g}) \\
\mathrm{MeetsX}(\bar{h})
\end{array} \wedge
\right) \rightarrow
\neg \left( \disj(\bar{g},\bar{h}) \wedge \left(
\begin{array}{r}
\mathrm{SamePD}(\bar{f_0},\bar{h})\\
\mathrm{SamePD}(\bar{f_1},\bar{h})
\end{array} \vee \right)\right)\right)
\end{array}
$$
\end{dfn}

\begin{prop}
$W \models \mathrm{RepPointDec}(f_0,f_1)$ if and only if $x_{f_0} = x_{f_1} \in X$.
\end{prop}
\begin{proof}
$\mathrm{RepPointDec}$ is only realised by tuples that satisfy $\mathrm{MeetsX}$, so Lemma 3.17 of Part II shows that this proposition is true.
\end{proof}

\begin{dfn}
$\mathrm{EquivRepPointDec}$, $\mathrm{Temp1PBDec}$, $\mathrm{Temp2PBDec}$, $\mathrm{PathBetweenDec}$, $\mathrm{RelatedDec}$ and $\mathrm{BDec}$ are the formulas $\mathrm{EquivRepPoint}$, $\mathrm{Temp1PB}$, $\mathrm{Temp2PB}$, $\mathrm{PathBetween}$, $\mathrm{Related}$ and $\mathrm{B}$ with every instance of $\mathrm{RepPoint}$ replaced by $\mathrm{RepPointDec}$.
\end{dfn}

\begin{theorem}\label{thm:reconX}
$(\mathrm{RepPointDec}, \mathrm{EquivRepPointDec}, \mathrm{BDec} )$ is a first order interpretation of $(X,B)$ inside $W$.
\end{theorem}
\begin{proof}
Since the other formulas in the interpretation only quantify over the points that realise $\mathrm{RepPointDec}$, the proofs of Section 3 of Part II apply directly.
\end{proof}

\subsection{Reconstructing $S$ and $(T,L)$}\label{subsection:reconSTL}

Now that we are able to refer to $X$ inside $W$, we can exploit this fact to define subgroups isomorphic to $\aut(S)$ and $\aut(T,L)$ inside $W$.  While the initial stages of the definitions are first order, I am unable to make the final jump without using second order logic.

\begin{dfn}
Let $\mathrm{FunctionPart}(\phi)$ be the formula
$$
\forall \bar{f}_0, \bar{f}_1, \bar{g}_0, \bar{g}_1 \left(
\begin{array}{c}
\left(\begin{array}{c}
\mathrm{RepPointDec}(\bar{f}_0, \bar{f}_1) \\
\mathrm{RepPointDec}(\bar{g}_0, \bar{g}_1) 
\end{array}
\wedge \right) \wedge\left(
\begin{array}{c}
(\bar{f}_0^\phi = \bar{g}_0 \wedge \bar{f}_1^\phi = \bar{g}_1) \\
(\bar{f}_1^\phi = \bar{g}_0 \wedge \bar{f}_0^\phi = \bar{g}_1) 
\end{array} \vee \right) \\
\rightarrow \mathrm{EquivRepPointDec}(\bar{f}_0, \bar{f}_1 ; \bar{g}_0, \bar{g}_1) 
\end{array}
\right)
$$
\end{dfn}

\begin{lemma}
$W \models \mathrm{FunctionPart}(\phi)$ if and only if $\phi$ fixes $X$ point-wise.  
\end{lemma}
\begin{proof}
$\phi$ fixes $X$ point-wise if and only if $\psi(x) = \psi^\phi(x)$ for all $x \in X$.
\end{proof}

\begin{prop}
$\mathrm{FunctionPart}(W) \cong  \prod\limits_{i \in X} \aut(S_i) \times \prod\limits_{(i,j) \in X_{ap}} \aut(T_{(i,j)},L_{(i,j)})$
\end{prop}
\begin{proof}
$\phi \in \mathrm{FunctionPart}(W)$ if and only if $\phi$ fixes $X$ point-wise, i.e. is of the form $(id, \eta, \zeta)$.
\end{proof}

\begin{dfn}
$\mathrm{AboveWitness}(\phi ; \bar{f}_0, \bar{f}_1)$ is the formula
$$\forall \bar{g}_0, \bar{g}_1 ( \mathrm{EquivRepPointDec}(\bar{g}_0,\bar{g}_1; \bar{g}_0^\phi , \bar{g}_1^\phi) \rightarrow \mathrm{EquivRepPointDec}(\bar{f}_0,\bar{f}_1; \bar{g}_0 , \bar{g}_1) )$$
\end{dfn}

\begin{dfn}
$\mathrm{BetweenWitness}(\phi ; \bar{f}_0, \bar{f}_1, \bar{g}_0, \bar{g}_1)$ is the formula
$$
\begin{array}{c}
\mathrm{RelatedDec}(f,g) \wedge (\forall h \neg \mathrm{PathBetweenDec}(h; f,g)) \wedge \\
\forall \bar{h}_0, \bar{h}_1 \left( \left.
\begin{array}{l}
\mathrm{EquivRepPointDec}(\bar{h}_0,\bar{h}_1; \bar{h}_0^\phi , \bar{h}_1^\phi) \rightarrow\\
\hspace{40pt} \left(
\begin{array}{l}
\mathrm{EquivRepPointDec}(\bar{f}_0,\bar{f}_1; \bar{h}_0 , \bar{h}_1) \\
\mathrm{EquivRepPointDec}(\bar{g}_0,\bar{g}_1; \bar{h}_0 , \bar{h}_1)
\end{array}
\vee \right) 
\end{array}
\right) \right)
\end{array}
$$
\end{dfn}

\begin{lemma}
If $W \models \mathrm{AboveWitness}(\phi ; \bar{f}_0, \bar{f}_1)$ then for all $g \in X$
$$\phi(g)=g \Leftrightarrow g=f$$
If $W \models \mathrm{BetweenWitness}(\phi ; \bar{f}_0, \bar{f}_1)$ then $f$ is either a successor or predecessor of $g$ and for all $h \in X$
$$\phi(h)=h \Leftrightarrow h=f \; \mathrm{or} \; h=g$$
\end{lemma}
\begin{proof}
This is follows from the fact that if $(\bar{f}_0,\bar{f}_1)$ represents $f$ then $(\bar{f}^\phi_0,\bar{f}^\phi_1)$ represents $\phi(f)$.
\end{proof}

Finally we resort to second order logic to define subgroups of $\mathrm{FunctionPart}(W)$ isomorphic to $\aut(S)$ and $\aut(T,L)$.

\begin{dfn}
\textcolor{white}{Gap}

\begin{enumerate}
\item $\mathrm{AboveTemp1}(A;f)$ is the second order formula

$$
A \lneqq \mathrm{FunctionPart}(W) \wedge
\forall \phi \left(
\begin{array}{rcl}
\mathrm{AboveWitness}(\phi ; f) &\rightarrow& \phi A = A\\
\end{array}\right)
$$
$\mathrm{AboveTemp2}(A;f)$ is the second order formula
$$
\begin{array}{c}
\mathrm{AboveTemp1}(A;f) \wedge \\
\forall B,C (( \mathrm{AboveTemp1}(B;f)\wedge \mathrm{AboveTemp1}(C,f)) \rightarrow BC \not=A)
\end{array}
$$
and $\mathrm{AboveTemp3}(A,f)$ is the formula
$$\begin{array}{c}
\mathrm{AboveTemp2}(A;f) \wedge \\
 \forall B\not= A (\mathrm{AboveTemp2}(B;f) \rightarrow \neg \exists \phi ( \phi(B) \leq A)
\end{array}$$
\item $\mathrm{BetweenTemp1}(A;f,g)$ is the second order formula
$$
A \lneqq \mathrm{FunctionPart}(W) \wedge
\forall \phi \left(
\begin{array}{rcl}
\mathrm{BetweenWitness}(\phi ; f,g) &\rightarrow& \phi A = A\\
\end{array}\right)
$$
and $\mathrm{BetweenTemp2}(A;f,g)$ is the second order formula
$$
\begin{array}{c}
\mathrm{BetweenTemp1}(A;f) \wedge \\
\forall B,C (( \mathrm{BetweenTemp1}(B;f,g)\wedge \mathrm{BetweenTemp1}(C,f)) \rightarrow BC \not=A)
\end{array}
$$
\item $\mathrm{Between}(A,f,g)$ is the second order formula
$$\begin{array}{c}
\mathrm{BetweenTemp2}(A;f,g) \wedge \\
\forall B\not=A (\mathrm{BetweenTemp2}(B;f,g) \rightarrow \neg \exists \phi ( \phi(B) \leq A))
\end{array}$$
$\mathrm{Above}(A,f)$ is the second order formula
$$
{AboveTemp3}(A,f) \wedge \forall B,g ( \mathrm{Between}(B;f,g) \rightarrow \neg (A \subset B))$$
\end{enumerate}
\end{dfn}

\begin{theorem}\textcolor{white}{Gap}\label{thm:reconSTL}

\begin{enumerate}
\item If $M \models \mathrm{Above}(A;f)$ then $A \cong \aut(S)$.
\item If $M \models \mathrm{Between}(A;f,g)$ then $A \cong \aut(T,L)$.
\end{enumerate}
\end{theorem}
\vspace{-10pt}
\begin{proof} Let $\pi_x$ and $\pi_{(x,y)}$ be the projection functions from
$$\prod\limits_{i \in X} \aut(S_i) \times \prod\limits_{(i,j) \in X_{ap}} \aut(T_{(i,j)},L_{(i,j)})$$
to $\aut(S_x)$ and $\aut(T_{(x,y)},L_{(x,y)})$ respectively.  Let $B$ be such that
$$W \models \mathrm{AboveTemp1}(B,f)$$
Since for all $\phi$ such that $W \models \mathrm{AboveWitness}(\phi ; f)$
$$\pi_x(B)=\pi_{\phi(x)}(B) \quad\mathrm{and}\quad \pi_{(x,y)}(B)=\pi_{(\phi(x),\phi(y))}(B)$$
then for any $\phi \in W$ we may obtain by patching a $\psi$ such that $\psi|_{S_f}=\phi|_{S_f}$ and
$$W \models \mathrm{AboveWitness}(\phi ; f)$$
and so for any $a \in \aut(S)$ there is a $\psi$ such that $\pi_{f}(\psi) =  a$, and since $A$ is a subgroup preserved under composition with $\psi$, we know that $a \in \pi_f(B)$.

Variations on this argument show that for all $x$
$$\pi_x(B)=\aut(S) \:\mathrm{or} \lbrace id \rbrace \quad\mathrm{and}\quad \pi_{(x,y)}(B)=\aut(T,L) \:\mathrm{or} \lbrace id \rbrace $$

Similarly, if $W \models \mathrm{BetweenTemp1}(B,f,g) \wedge \mathrm{BetweenWitness}(\phi ; f,g)$ then 
$$\pi_x(B)=\pi_{\phi(x)}(B) \quad\mathrm{and}\quad \pi_{(x,y)}(B)=\pi_{(\phi(x),\phi(y))}(B)$$
and
$$\pi_x(B)=\aut(S) \:\mathrm{or} \lbrace id \rbrace \quad\mathrm{and}\quad \pi_{(x,y)}(B)=\aut(T,L) \:\mathrm{or} \lbrace id \rbrace $$

If $W \models \mathrm{AboveTemp2}(A,f)$ then $A$ cannot be realised as the composition of two subgroups that satisfy $\mathrm{AboveTemp1}$ and so if $A$ is not the identity on $S_x$ then $A$ is the identity on all the $T_{(z_0,z_1)}$, and is not the identity on $S_y$ if and only if
$$\exists \phi \in \aut_{f}(W) \: \phi(x)=y$$
If $A$ is not the identity on $T_{(z_0,z_1)}$ then $A$ is the identity on all the $S_x$, and is not the identity on $T_{(y_0,y_1)}$ if and only if
$$\exists \phi \in \aut_{f}(W) \: \phi((z_0,z_1))=(y_0,y_1)$$

Similarly if $W \models \mathrm{BetweenTemp2}(A,f,g)$ then if $A$ is not the identity on $S_x$ then $A$ is the identity on all the $T_{(z_0,z_1)}$, and is not the identity on $S_y$ if and only if
$$\exists \phi \in \aut_{(f,g)}(W) \: \phi(x)=y$$
If $A$ is not the identity on $T_{(z_0,z_1)}$ then $A$ is the identity on all the $S_x$, and is not the identity on $T_{(y_0,y_1)}$ if and only if
$$\exists \phi \in \aut_{(f,g)}(W) \: \phi((z_0,z_1))=(y_0,y_1)$$

If $W \models \mathrm{AboveTemp3}(A,f)$ then given any $B \not=A$ that satisfies $ \mathrm{AboveTemp2}(A,f,g)$, we are unable to map $B$ into $A$ using members of $W$ (other embeddings may exists, but not inside $W$).  This means that either $A$ does not act as the identity on $S_f$ only, or $A$ does not act as the identity on $\bigcup T_{(f,g)}$.

\paragraph{}

We now examine $\mathrm{Between}(A,f,g)$, and we suppose that $W \models \mathrm{Between}(A,f,g)$.  The only family that does not permit $B$ to satisfy $\mathrm{BetweenTemp2}$ is the one that only acts non-trivially on $T_{(f,g)}$.  Therefore
$$W \models \mathrm{Between}(A,f,g) \Rightarrow A \cong \aut(T,L)$$

If $W \models \mathrm{Above}(A,f)$ then $A$ does not contain any subset that satisfies $\mathrm{Between}$, so $A \cong \aut(S)$.
\end{proof}

\section{Final Results}\label{sec:final}

\begin{dfn}
Let $Z$ be the one element partial order.
$$K_{Dec} := \left\lbrace M \: : \:
\begin{array}{c}
\exists X \in K_{Cone} \cup \lbrace \emptyset, Z \rbrace \; \exists S, (T,L) \in K_{Rub}\cup \lbrace \emptyset \rbrace \\
M = \mathrm{Dec}(X,S,(T,L))
\end{array}
\right\rbrace$$
Note that $\mathrm{Dec}(Z,S,(T,L))$ equals $S$ if $S$ is non-empty, or $Z$ if $S$ is empty.  $\mathrm{Dec}(\emptyset,S,(T,L))$ is the empty partial order for any $S$ and $(T,L)$, and $\mathrm{Dec}(X,\emptyset,\emptyset) = X$ for any $X \in K_{Cone}$.
\end{dfn}

We allow $Z$ and $\emptyset$ as arguments in $\mathrm{Dec}(X,S,(T,L))$ to ensure that $K_{Cone}, K_{Rub} \subseteq K_{Dec}$.

\begin{theorem}
$K_{Dec}$ is faithful.
\end{theorem}
\begin{proof}
Let $ \mathrm{Dec}(X_0,S_0,(T_0,L_0)), \mathrm{Dec}(X_1,S_1,(T_1,L_1)) \in K_{Dec}$ and assume that
$$\aut(\mathrm{Dec}(X_0,S_0,(T_0,L_0))) \cong \aut(\mathrm{Dec}(X_1,S_1,(T_1,L_1))) $$
Theorem \ref{thm:reconX} shows that $(X_0,B) \cong (X_1,B)$.  For all $M \in K_{Cone}$
$$M^* \in K_{Cone} \Rightarrow M \cong M^*$$
Therefore $X_0 \cong X_1$.

Theorem \ref{thm:reconSTL} shows that $S_0 \cong S_1$ and $(T_0,L_0) \cong (T_1,L_1)$.
\end{proof}

\begin{theorem}
Let $M$ be a CFPO, let $A$ be a 1-orbit such that $\aut(M)$ acts cone transitively on $A$, and for any $B \subset M$ let $ \sim_B$ be the equivalence relation $x \sim y \Leftrightarrow \path{x,y} \cap B = \emptyset$.  We let $C \in (M \setminus A) / \sim_A$, and describe two conditions.
\begin{enumerate}
\item If $\path{C,M \setminus C} \not= \emptyset$ then there is an $a_c \in A$ such that $$\path{C,M \setminus C} = \lbrace a_C \rbrace$$

This says that if there is only one way to go from $C$ to $M \setminus C$ then $C$ is attached to $a_c$.
\item If $\path{C,M \setminus C} = \emptyset$ then:
\begin{enumerate}
\item $(M \setminus C) / \sim_C$ has exactly two elements which we call $B_C$ and $B'_C$; and
\item there is $(a_C, a'_C) \in A_{ap}$ such that
$$
\begin{array}{c c c}
\path{C,B_C} = \lbrace a_C \rbrace & \textnormal{ and } & \path{C,B'_C} = \lbrace a'_C \rbrace
\end{array}
$$
\end{enumerate}
This says that if there is more than one way to go from $C$ to $M \setminus C$ then $C$ lies between an adjacent pair of $A$.
\end{enumerate}
If every $C \in (M \setminus A) / \sim_A$ satisfy both 1. and 2. then there is an $N \in K_{Dec}$ such that $\aut(M) \cong_A \aut(N)$.
\end{theorem}

\section{Unresolved Questions}

The least acceptable assumption made in these papers is in Paper II, that both $ro \uparrow (M)$ and $ro \downarrow (M)$ are at least than 5.  It is a somewhat artificial condition, can it be eliminated?

\begin{question}
Is there an interpretation that works for cone transitive CFPOs where at least one of $ro \uparrow (M)$ and $ro \downarrow (M)$ is less than 5?
\end{question}

The second transitivity condition of Paper II is both strong and unnatural; simply assuming 1-transitivity is much weaker.  In her Ph.D. thesis, Chicot gives a classification of the countable 1-transitive trees \cite{Chicot2004}.  It is an impressive result; there are $2^{\aleph_0}$ many, and they are extremely wild.  They may even have multiple non-isomorphic maximal branches, which are not even 1-transitive!

The maximal branches do have to be `lower isomorphic', i.e. any two principal initial sections of any two maximal branches of a 1-transitive tree must be isomorphic.  This suggests to me that the maximal chains of some 1-transitive CFPOs may be only `interval isomorphic', meaning that any two intervals of the maximal chains are isomorphic.

It would be a wonderful thing to reconstruct the 1-transitive CFPOs.  The frustrating thing is that this second condition is so necessary to the method that I don't believe there is a way to eliminate it.  How can one use the subgroups isomorphic to $A_5$ without assuming that there are any?

Nonetheless, this presents a project:

\begin{question}
Classify the (countable) 1-transitive CFPOs.
\end{question}

Perhaps a method for reconstruction would present itself if they were better understood, but the classification of the 1-transitive trees was an impressive feat, so a classification of the 1-transitive CFPOs would be difficult. A more modest objective would be to find examples that reject our methods entirely.

\begin{question}
Give an example of a 1-transitive CFPO where $\aut(M)$ is unable to act as $A_5$ on the cones of a point, but $ro \uparrow (M)$ and $ro \downarrow (M)$ are greater than or equal to 5.
\end{question}

Even if we had a reconstruction of the class of 1-transitive CFPOs, we would not be able to use decoration to reconstruct the whole class of CFPOs.

\begin{example}
$W(Alt,\mathbb{Z},\emptyset)$ is not the automorphism group of a tree, nor a 1-transitive CFPO, nor the automorphism group of the decoration of a 1-transitive CFPO by  trees.
\end{example}
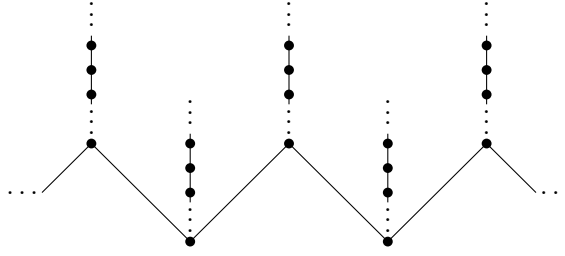
\begin{figure}
\begin{center}
\begin{tikzpicture}[scale=0.13]
\draw (-25,0) -- (-20,5) -- (-10,-5) -- (0,5) -- (10,-5) -- (20,5) -- (25,0);

\fill (-20,5) circle (0.5);
\fill (-10,-5) circle (0.5);
\fill (0,5) circle (0.5);
\fill (20,5) circle (0.5);
\fill (10,-5) circle (0.5);

\fill(-20,10) circle (0.5);
\fill(-20,12.5) circle (0.5);
\fill(-20,15) circle (0.5);
\draw (-20,9) -- (-20,16);
\fill(0,10) circle (0.5);
\fill(0,12.5) circle (0.5);
\fill(0,15) circle (0.5);
\draw (0,9) -- (0,16);
\fill(20,10) circle (0.5);
\fill(20,12.5) circle (0.5);
\fill(20,15) circle (0.5);
\draw (20,9) -- (20,16);

\draw[anchor=north] (-20,11) node {$\vdots$};
\draw[anchor=south] (-20,16) node {$\vdots$};
\draw[anchor=north] (0,11) node {$\vdots$};
\draw[anchor=south] (0,16) node {$\vdots$};
\draw[anchor=north] (20,11) node {$\vdots$};
\draw[anchor=south] (20,16) node {$\vdots$};

\fill(-10,0) circle (0.5);
\fill(-10,2.5) circle (0.5);
\fill(-10,5) circle (0.5);
\draw (-10,-1) -- (-10,6);
\fill(10,0) circle (0.5);
\fill(10,2.5) circle (0.5);
\fill(10,5) circle (0.5);
\draw (10,-1) -- (10,6);

\draw[anchor=north] (-10,1) node {$\vdots$};
\draw[anchor=south] (-10,6) node {$\vdots$};
\draw[anchor=north] (10,1) node {$\vdots$};
\draw[anchor=south] (10,6) node {$\vdots$};

\draw (-27,0) node {$\ldots$};
\draw (27,0) node {$\ldots$};

\end{tikzpicture}
\caption{$Dec(Alt,\mathbb{Z},\emptyset)$}
\end{center}
\end{figure}

Which informs the next question:

\begin{question}
Is there a minimal class of CFPOs such that every automorphism group of a CFPO occurs as the automorphism group of a decoration of a member of the class by trees?
\end{question}

In \cite{Parigot1982}, as well as showing that all completions of the theory of trees are NIP, Parigot shows that the theory of a tree is stable if and only if every maximal branch has at most $n$ elements for some $n \in \mathbb{N}$.

While I am almost certain that this is also for the CFPOs, there is perhaps scope for defining an infinite order even when the maximal branches are finite, for example in Alt the pairs $(a_0, a_{2n})$ have a natural order.  While I would be shocked if this order is definable, I cannot see a way to prove that it is undefinable in all CFPOs.

\begin{question}
Is a CFPO stable if and only if all its maximal branches have at most $n$ elements for some fixed $n \in \mathbb{N}$.
\end{question}

\bibliography{bib}{}
\bibliographystyle{plain}

\end{document}